\documentclass[12pt]{article}
\usepackage{amsmath,enumerate,amsfonts,color,amssymb,amsthm, verbatim}

\usepackage{hyperref}
\usepackage[normalem]{ulem}

\usepackage{tikz}

\def\ZZ{{\mathbb Z}}
\def\RR{{\mathbb R}}
\def\Sphere{{\mathbb S}}

\def\vareps{{\varepsilon}}

\newtheorem{theorem} {\sc  Theorem\rm} [section]

\newtheorem{corollary} [theorem] {\sc  Corollary\rm}
\newtheorem{lemma} [theorem] {\sc  Lemma\rm}
\newtheorem{proposition} [theorem] {\sc  Proposition\rm}

\newtheorem{definition}[theorem]{\sc  Definition\rm}

\newcounter{marnote}

\DeclareFontFamily{OT1}{rsfs}{}
\DeclareFontShape{OT1}{rsfs}{m}{n}{ <-7> rsfs5 <7-10> rsfs7 <10-> rsfs10}{}
\DeclareMathAlphabet{\mycal}{OT1}{rsfs}{m}{n}

\def\be{\begin{equation}}
\def\ee{\end{equation}}

\def\be{\begin{equation}}
\def\ee{\end{equation}}
\def\bea#1\eea{\begin{align}#1\end{align}}

\begin{document}
\title{Zero-homogeneous and $O(2)$-equivariant\\critical points of the Oseen-Frank energy\\ with multiple Frank constants}

\author{Luc Nguyen\thanks{Mathematical Institute and St Edmund Hall, University of Oxford, Andrew Wiles Building, Radcliffe Observatory Quarter, Woodstock Road, Oxford OX2 6GG, United Kingdom. Email: luc.nguyen@maths.ox.ac.uk.}}

\date{}

\maketitle
\begin{abstract}
We give an existence and classification result for zero-homogeneous and $O(2)$-equivariant critical points of the Oseen-Frank energy with multiple Frank constants. These critical points carry a topological defect of degree one and are minimizing with respect to $O(2)$-equivariant perturbations supported away from the axis of symmetry.



\medskip
\centerline{Dedicated to Professor Dang Duc Trong with admiration.}
\end{abstract}

\tableofcontents

\section{Introduction}

In this note, we construct a family of zero-homogeneous and $O(2)$-equivariant equilibrium configurations in the Oseen-Frank theory of nematic liquid crystals with multiple Frank constants.

Recall that, in the Oseen-Frank theory of nematic liquid crystals, equilibrium configurations are critical points of the Oseen-Frank energy
\begin{equation}
E[u;\Omega] = \int_\Omega W(u,\nabla u)\,dx
	\label{Eq:OFEnergy}
\end{equation}
where $\Omega \subset \RR^3$ is a bounded domain, $u \in H^1(\Omega, \Sphere^2)$ and the free energy density $W$ takes the form
\begin{align}
2W(u,\nabla u) 
	&= k_1 (\nabla \cdot u)^2 + k_2( u \cdot (\nabla \times u))^2 + k_3 |u \times (\nabla \times u)|^2\nonumber\\
		&\qquad + (k_2+k_4)\big[\mathrm{tr}((\nabla u)^2) - (\nabla \cdot u)^2\big].
		\label{Eq:Wdef}
\end{align}
Here $k_1 > 0, k_2 > 0, k_3 > 0$ and $k_4 \in \RR$ are the Frank (elastic) constants. We refer the reader to the survey \cite{BallSurvey} for an account of mathematical results on the Oseen-Frank theory as well as other theories of liquid crystals.

As is well known, the term $\mathrm{tr}((\nabla u)^2) - (\nabla \cdot u)^2$ in the expression of $W$ is a null Lagrangian. Therefore, whenever a Dirichlet boundary $u|_{\partial\Omega} = u_0$ is given, its contribution to $E$ is a fixed constant and can be discarded in the study of critical points of $E$. In other words, one may choose $k_4$ to be any constant as one wishes. In particular, if we write $\alpha = \min\{k_1,k_2,k_3\}$ and $\beta = 3k_1 + 2k_2 + 2k_3$, then with the choice $k_4 = \alpha - k_2$, it holds that
\[
\frac{\alpha}{2}  |\nabla u|^2 \leq W(u,\nabla u) \leq \frac{\beta}{2} |\nabla u|^2,
\]
that is, $W$ is coercive (see e.g. \cite[equation (2.9)]{AlGhi97-AnnIHP}). A standard application of the direct method of the calculus of variations shows that, for any given $u_0 \in H^1(\Omega,\Sphere^2)$, there exists a minimizer of $E$ in $H^1_{u_0}(\Omega,\Sphere^2) = \{u \in H^1(\Omega,\Sphere^2): u = u_0 \text{ on } \partial\Omega\}$.

A map $u \in H^1(\Omega,\Sphere^2)$ is a critical point for $E$ if it satisfies
\begin{multline}
-k_1 \nabla (\nabla \cdot u) + k_3 \nabla \times (\nabla \times u)\\
	 + (k_2 - k_3) \Big[ \nabla \times((u \cdot (\nabla \times u)) u) + (u \cdot (\nabla \times u))\nabla \times u\Big] = \lambda u
	 \label{Eq:ELFull}
\end{multline}
in $\Omega$, where $\lambda$ is a Lagrange multiplier accounting for the unimodular constraint.

The energy density $W$ is rotationally equivariant, that is, 
\[
W(Ru, R M R^T) = W(u,M) \text{ for all } u \in \RR^3, M \in \RR^{3 \times 3}, R \in O(3).
\]
This can be used to construct equivariant critical points of $E$. For example, if $G$ is a subgroup of $O(3)$ and the domain $\Omega$ and the boundary data $u_0$ are $G$-equivariant in the sense that
\[
\begin{cases}
 Rx \in \Omega \quad \forall~x \in \Omega,\\
 u_0(Rx) = R u_0(x)
\quad \forall~x \in \partial\Omega, R \in G,
\end{cases}
\]
then $E$ has a $G$-equivariant critical point $u$:
\[
\begin{cases}
u \text{ is a critical point of } E[\cdot,\Omega],\\
u|_{\partial\Omega} = u_0,\\
u(Rx) = Ru(x) \quad\forall~x \in  \Omega, R \in G.
\end{cases}
\]
For example, such critical point can be obtained by minimizing $E$ among $G$-equivariant maps.

A special case which has been studied extensively is the case of $O(3)$-equivariant (or more casually, rotationally symmetric) critical points. It is not hard to show that a map $u \in H^1(\Omega,\mathbb{S}^2)$  on an $O(3)$-equivariant domain $\Omega$ is $O(3)$-equivariant if and only if  
\[
u(x) = \pm n(x) := \pm \frac{x}{|x|} \quad \forall~x \in \Omega.
\]
The map $n(x) = \frac{x}{|x|}$ is frequently referred to as the hedgehog in the literature. It is a fact that the hedgehog $n$ is a critical point for $E$ for every choice of Frank constants. 

Beside being $O(3)$-equivariant, the hedgehog is also zero-homogeneous and thus can appear as a tangent map when one performs a blow-up analysis near a singular point of critical points. To motivate our later discussion, let us briefly recall some known results in the best established case where $k_1 = k_2 = k_3$ (that is, the `one-constant' case). In this case, with $k_4 = k_1 - k_2$, the energy $E$ is the Dirichlet energy, and minimizers of $E$ are minimizing harmonic maps. By a result of Schoen and Uhlenbeck \cite{su}, away from $\partial\Omega$, minimizing harmonic maps are regular away from a finite set of `point defects' and the tangent map near a defect is zero-homogeneous. Moreover, Brezis, Coron and Lieb showed in \cite{BrezisCoronLieb} that, modulo a multiplication by a constant orthogonal matrix, the tangent map of a minimizing harmonic map near a point defect is given by the hedgehog. In particular, they showed that the hedgehog $n$ is the unique minimizer for $E$ with respect to its own boundary data. A different proof of this statement was given by Lin \cite{Lin-CR87}. We however note that the hedgehog is not the unique zero-homogeneous harmonic map. In spherical coordinates 
\[
x = (r\sin\theta \cos\varphi, r \sin\theta \sin\varphi, r\cos\theta),
\]
a special family of zero-homogeneous harmonic maps is given by (see e.g. Poon \cite[p. 167]{Poon-JDG91}\footnote{The expression in \cite{Poon-JDG91} can be put in the stated form with $t = \arccos \tanh c$.})
\begin{equation}
n_t(r,\theta,\varphi) = \Big(\frac{\sin t \sin \theta}{1 + \cos t \cos\theta} \cos\varphi , \frac{\sin t \sin \theta}{1 + \cos t \cos\theta}  \sin\varphi, \frac{\cos t + \cos \theta}{1 + \cos t \cos\theta} \Big)^T,
	\label{Eq:HMnt}
\end{equation}
where $t \in (0,\pi)$ is a parameter. It is readily seen that the hedgehog $n = n_{\pi/2}$ is a member of this family. However, as mentioned above, these maps are not minimizing harmonic maps except for the hedgehog $n = n_{\pi/2}$. 

The regularity of minimizers of the Oseen-Frank energy when the Frank constants are different is much less understood. The main result in this direction is due to Hardt, Kinderlehrer and Lin \cite{Hardt-Kinder-Lin}: every minimizer is regular away from a set of zero one-dimensional Hausdorff measure. See also Hong \cite{Hong04}. There is a lack of understanding of defect structure in general, including the simplest case of a point defect. Available results point toward the minimality or stability of the hedgehog:
\begin{enumerate}[(i)]
\item When $k_2 \geq k_1$, H\'elein observed in \cite{Helein-CR87} that F.-H. Lin's proof in \cite{Lin-CR87} could be use to show that $n$ is the unique minimizer of $E$ subjected to its own boundary data (see also Ou \cite{Ou-JGA92}). 
\item When $k_1$ is sufficiently large, S.-Y. Lin showed in \cite{LinSY-Thesis} that $n$ is no longer minimizing for $E$.
\item When $8(k_2 - k_1) + k_3 < 0$, H\'elein showed in \cite{Helein-CR87} that the second variation of $E$ at $n$ is negative along certain direction and so $n$ is not minimizing for $E$. 
\item When $8(k_2 - k_1) + k_3 \geq 0$, Cohen and Taylor \cite{CohenTaylor-CPDE90} showed that $n$ is a local minimum for $E$ along any curve passing through $n$ in $H^1_n(\Omega,\Sphere^2) \cap (W^{1,p} \cap L^q)(\Omega,\RR^3)$ which is $C^2$ with respect to $(W^{1,p} \cap L^q)(\Omega,\RR^3)$-topology provided $p \in (2,3)$, $q \in (2,\infty)$ and $\frac{1}{p} + \frac{1}{q} \leq \frac{1}{2}$ (compare Kinderlehrer and Ou \cite{KinderOu-PRSLA92}). 
\end{enumerate}
As a matter of fact, experimental data for many popular nematic liquid crystals such as PAA (\cite[p. 104]{dg}, \cite[Table I]{DeJeuCS-76}), MBBA (\cite[p. 105]{dg}, \cite[Table II]{DeJeuCS-76}), 7CB, 7PCH, 5C (\cite[Table II]{SchadOsman-81}) indicate that $8(k_2 - k_1) + k_3 < 0$ in most instances. In such cases, the hedgehog is not minimizing and there is no known tangent map for a minimizer near a point defect (if a point defect does exist, say).

The goal of this note is related to the above gap of understanding. We construct a family of zero-homogeneous critical points of $E$ which generalizes the family \eqref{Eq:HMnt}. We hope that some of these new critical points would be a possible tangent map for a minimizer near a point defect.

The starting point is that the family \eqref{Eq:HMnt} is not only zero-homogeneous but also $O(2)$-equivariant in the sense that
\[
n_t(Rx) = Rn_t(x) \text{ for all } R \in \{A \in O(3): Ae_3 = e_3\} \cong O(2),
\]
where $e_3 = (0, 0, 1)^T$. We are thus led to consider $O(2)$-equivariant maps, which, in spherical coordinates, take the form (see Lemma \ref{Lem:O2Lift})
\begin{align}
u(r,\theta,\varphi) 
	&= \sin\psi(r,\theta) e_\varphi^\perp + \cos \psi(r,\theta) e_3, \quad \text{ for } r < 0 \leq 1 \text{ and } 0 \leq \theta \leq \pi,
	\label{Eq:Ansatz1}
\end{align}
where
\begin{equation}
e_\varphi = (-\sin\varphi, \cos\varphi, 0)^T \quad \text{ and } \quad 		e_\varphi^\perp = (\cos\varphi, \sin\varphi, 0)^T.
	\label{Eq:epeppDef}
\end{equation}
We note that $O(2)$-equivariance is stronger than $SO(2)$-equivariance and the latter can be more casually referred to as axial symmetry. We will not be concerned with $SO(2)$-equivariant maps in this note.

We establish:

\begin{theorem}\label{Thm:Main}
Let $k_1, k_2, k_3 > 0$. We have the following conclusions.
\begin{enumerate}[(i)]
\item For every $t \in (0,\pi)$, there exists a unique solution $\psi_t \in C^\infty((0,\pi)) \cap C^{2,1}([0,\pi])$ of the problem
\begin{equation}
\begin{cases}
\displaystyle \psi_t' = \frac{k_1^{1/2} \sin\psi_t}{\sin \theta \big( k_1 \cos^2(\psi_t - \theta) + k_3 \sin^2(\psi_t - \theta)\big)^{1/2}} &\text{ in } (0,\pi),\\
\psi_t(\pi/2) = t.
\end{cases}
	\label{Eq:psitdef}
\end{equation}
Moreover, $\psi_t$ is increasing, $\psi_t(0) = 0, \psi_t(\pi) = \pi$ and $\frac{\sin\psi_t}{\sin\theta} \in C^{1,1}([0,\pi])$.

\item For every $t \in (0,\pi)$, the map $u_t$ defined by 
\[
u_t(r,\theta,\varphi) = (\sin \psi_t(\theta) \cos \varphi, \sin \psi_t(\theta) \sin \varphi, \cos \psi_t(\theta))^T
\]
belongs to $H^1_{\rm loc}(\RR^3) \cap C^1(\RR^3 \setminus \{0\})$, has a singularity around the origin of topological degree one and is a zero-homogeneous and $O(2)$-equivariant critical point of the Oseen-Frank energy $E$. Moreover, every zero-homogeneous and $O(2)$-equivariant critical point of $E$ in $H^1_{\rm loc}(\RR^3) \cap C^1(\RR^3 \setminus \{0\})$ which has a singularity around the origin of topological degree one must coincide with some such $u_t$.

\item When $t = \pi/2$, $\psi_{\pi/2}(\theta) = \theta$ and $u_{\pi/2} = n$. When $k_1 = k_3$, $u_t = n_t$ for all $t \in (0,\pi)$.

\item If $\Omega$ is a ball centered at the origin, then, for every $t \in (0,\pi)$ and every $O(2)$-equivariant maps $u \in C^1(\Omega \setminus \{0\},\mathbb{S}^2)$ which coincides with $u_t$ in a neighborhood of $\partial \Omega \cup \{0\}$, it holds that
\[
E[u;\Omega] \geq E[u_t;\Omega].
\]
Equality holds if and only if $u = u_t$.
(See Proposition \ref{Prop:Min} for a stronger statement.)

\end{enumerate}
\end{theorem}

It should be clear that the maps $u_t$ do not depend on the value of $k_2$.

We remark that there is no zero-homogeneous $O(2)$-equivariant critical point in $H^1_{\rm loc}(\RR^3) \cap C^1(\RR^3 \setminus \{0\})$ of the Oseen-Frank energy $E$ of the form \eqref{Eq:Ansatz1} such that $\psi(0) = 0$ and $\psi(\pi) = j\pi$ with $j \notin \{0,\pm 1\}$. Imposing $\psi(\pi) = 0$ leads to the constant map $u \equiv e_3$ and imposing $\psi(\pi) = -\pi$ leads to $\psi = -\psi_t$ for some $t \in (0,\pi)$. See Corollary \ref{Cor:ODE2}.

It remains an interesting open problem if any of the map $u_t$ is minimizing for $E$ with respect to its own boundary data in the class of non-equivariant maps. If this is the case, those $u_t$ can arise as a tangent map near a point defect of a minimizer.

\subsubsection*{Rights retention statement.} For the purpose of Open Access, the author has applied a CC BY public copyright licence to any Author Accepted Manuscript (AAM) version arising from this submission.

\subsubsection*{Acknowledgment.} The author would like to thank Prof. Epifanio Virga for useful comments.

\section{Proof}

We assume that $\Omega$ is the unit ball and use spherical coordinates 
\[
x = (r\sin\theta \cos\varphi, r \sin\theta \sin\varphi, r\cos\theta).
\]
Let
\begin{align*}
D &= \{(r,\theta): 0 < r < 1, 0 < \theta < \pi\}.
\end{align*}
We assume throughout this section that $k_4 = - k_2$.

\subsection{$O(2)$-equivariant maps}

\begin{definition}
A map $u \in L^1(\Omega,\RR^3)$ is said to be $O(2)$-equivariant if for every matrix $R \in O(3)$ satisfying $Re_3 = e_3$ it holds that
\[
u(Rx) = R u(x) \text{ a.e. in } \Omega.
\]
\end{definition}

If we let 
\begin{align*}
R_\eta &= \begin{pmatrix}
	\cos \eta & -\sin \eta & 0\\
	\sin\eta & \cos \eta & 0\\
	0 &0 & 1
\end{pmatrix},\\
J 
	&= \begin{pmatrix}
	-1 & 0 & 0\\
	0 & 1 & 0\\
	0 &0 & 1
\end{pmatrix},
\end{align*}
then it is easy to see that $u \in L^1(\Omega,\RR^3)$ is $O(2)$-equivariant if and only if it satisfies simultaneously
\begin{equation}
u(R_\eta x) = R_\eta u(x) \text{ for almost every } x \in \Omega \text{ and for all } \eta \in [0,2\pi)
	\label{Eq:O2a}
\end{equation}
and
\begin{equation}
u(Jx) = Ju(x) \text{ for almost every } x \in \Omega.
\label{Eq:O2b}
\end{equation}

\begin{lemma}\label{Lem:O2Lift}
Suppose that $u \in H^1(\Omega,\mathbb{S}^2)$ is $O(2)$-equivariant. Then there exists $\psi \in H^1_{\rm loc}(D)$ such that \eqref{Eq:Ansatz1} holds a.e. in $\Omega$ and
\[
\int_D \Big[(\partial_r \psi)^2   + \frac{1}{r^2} (\partial_\theta \psi)^2 + \frac{1}{r^2 \sin^2 \theta}\sin^2\psi\Big] r^2 \sin\theta \,d\theta\,dr = \frac{1}{2\pi}\int_\Omega |\nabla u|^2\,dx < \infty.
\]
\end{lemma}

We will refer to the function $\psi$ in the conclusion of the lemma as the lifting of the $O(2)$-equivariant map $u$.

\begin{proof}
We adapt the proof of \cite[Proposition 2.1]{INSZ_AnnIHP}. We work in spherical coordinates. Let
\[
a = u \cdot e_\varphi^\perp, \quad b = u \cdot e_3, \quad c = u \cdot e_\varphi,
\]
where the vector fields $e_\varphi$ and $e_\varphi^\perp$ are defined by \eqref{Eq:epeppDef}. Then \eqref{Eq:O2a} implies that for every $\eta \in [0,2\pi)$, 
\[
a(r,\theta,\varphi) = a(r,\theta, \varphi + \eta), \quad b(r,\theta,\varphi) = b(r,\theta, \varphi + \eta), \quad c(r,\theta,\varphi) = c(r,\theta, \varphi + \eta)
\]
for a.e. $r \in (0,1), \theta \in (0,\pi), \varphi \in (0,2\pi)$. In particular, when $\frac{\eta}{2\pi}$ is irrational, this implies that $a, b, c$ are independent of $\varphi$:
\[
a = a(r,\theta), \quad b = b(r,\theta), \quad c = c(r,\theta).
\]
Noting that $e_\varphi \circ J = - J e_\varphi$ we deduce from \eqref{Eq:O2b} that
\[
c \equiv 0 \text{ and hence } a^2 + b^2 \equiv 1.
\]
Summarizing, we have shown that
\[
u(r,\theta,\varphi) 
	= \big(a(r,\theta)  \cos\varphi, a(r,\theta) \sin \varphi, b(r,\theta)\big)^T
\]
for $r < 0 \leq 1, 0 \leq \theta \leq \pi$ and $0 \leq \varphi < 2\pi$.

Next, we compute
\begin{align*}
|\nabla u|^2
	&= |\partial_r u|^2 + \frac{1}{r^2} |\partial_\theta u|^2 + \frac{1}{r^2 \sin^2 \theta} |\partial_\varphi u|^2\\
	&= (\partial_r a)^2 + (\partial_r b)^2 + \frac{1}{r^2} (\partial_\theta a)^2 + \frac{1}{r^2}(\partial_\theta b)^2 + \frac{1}{r^2 \sin^2 \theta}a^2.
\end{align*}
It follows that $(a,b) \in H^1_{\rm loc}(D,\mathbb{S}^1)$. By the lifting property for unimodular Sobolev maps (see \cite{BethuelZheng-JFA88, BBM-JAM00}), we deduce that there exists $\psi \in H^1_{\rm loc}(D)$ such that
\[
\begin{cases}
(\partial_r \psi, \partial_\theta \psi)
	= (-a \partial_r b + b \partial_r a, -a \partial_\theta b + b \partial_\theta a),\\
a(r,\theta) 
	= \sin \psi(r,\theta) \text{ and } b(r,\theta) = \cos \psi(r,\theta) \text{ a.e. in } D.
\end{cases}
\]
This gives the representation \eqref{Eq:Ansatz1}. Finally, we have
\begin{align*}
|\nabla u|^2
	&= (\partial_r \psi)^2   + \frac{1}{r^2} (\partial_\theta \psi)^2 + \frac{1}{r^2 \sin^2 \theta}\sin^2\psi,
\end{align*}
from which the last conclusion follows.
\end{proof}

Motivated by Lemma \ref{Lem:O2Lift}, we define
\begin{multline}
X(D) = \Big\{\psi \in H^1_{\rm loc}(D) \text{ such that }\\
	  \int_D \Big[(\partial_r \psi)^2   + \frac{1}{r^2} (\partial_\theta \psi)^2 + \frac{1}{r^2 \sin^2 \theta}\sin^2\psi\Big] r^2 \sin\theta \,d\theta\,dr < \infty \Big\}.
	  \label{Eq:XDdef}
\end{multline}
 
\begin{lemma}\label{Lem:XjAxis}
Let $\psi \in X(D)$. Then $\cos \psi \in W^{1,1}((r_0,1) \times (0,\pi))$ for any $r_0 \in (0,1)$. Moreover,
\[
\cos^2 \psi = 1 \text{ on } [0,1] \times \{0,\pi\}
\]
in the sense of trace.
\end{lemma}

\begin{proof}
We have by the definition of $X(D)$ and Cauchy-Schwarz' inequality that
\[
\int_D \Big[  |\sin\psi||\partial_r \psi|   + \frac{1}{r } |\sin\psi||\partial_\theta \psi|\Big]  r \,d\theta\,dr < \infty.
\]
This shows that $\cos \psi \in W^{1,1}((r_0,1) \times (0,\pi))$ for any $r_0 \in (0,1)$.

Fix some $r_0 \in (0,1)$.
For any $0 \leq \theta \leq \delta \leq \pi$, we have
\begin{align*}
\int_{r_0}^1 |\cos^2 \psi(r,\theta) - \cos^2 \psi(r,0)|\,dr
	&\leq 2\int_{r_0}^1 |\cos \psi(r,\theta) - \cos \psi(r,0)|\,dr\\
	&\leq 2\int_{[r_0,1] \times [0,\delta]} |\sin \psi(r,t)| |\partial_\theta \psi(r,t)|\,dt\,dr.
\end{align*}
It follows that
\begin{align*}
& \ln \frac{\tan(\delta/2)}{\tan(\delta/4)}  \int_{r_0}^1  (1 - \cos^2\psi (r,0))\,dr\\
	&\qquad=  \int_{[r_0,1] \times [\delta/2,\delta]}  \frac{1}{\sin\theta} (1 - \cos^2\psi (r,0))\,d\theta\,dr\\
	&\qquad \leq \int_{[r_0,1] \times [\delta/2,\delta]}  \frac{1}{\sin\theta} \sin^2\psi(r,\theta)\,d\theta\,dr\\
		&\qquad\qquad + \int_{[r_0,1] \times [\delta/2,\delta]}  \frac{1}{\sin\theta} |\cos^2\psi(r,\theta) - \cos^2\psi (r,0)|\,d\theta\,dr\\
	&\qquad \leq \int_{[r_0,1] \times [\delta/2,\delta]}  \frac{1}{\sin\theta} \sin^2\psi(r,\theta)\,d\theta\,dr\\
		&\qquad\qquad + 2\ln \frac{\tan(\delta/2)}{\tan(\delta/4)} \int_{[r_0,1] \times [0,\delta]} |\sin \psi(r,\theta)| |\partial_\theta \psi(r,\theta)|\,d\theta\,dr.
\end{align*}
Recalling that $\frac{1}{\sin\theta} \sin^2\psi(r,\theta)$ is intergrable over $D$, we deduce upon sending $\delta \rightarrow 0$ that
\[
\int_{r_0}^1  (1 - \cos^2\psi (r,0))\,dr = 0,
\]
from which it follows that $\cos^2\psi = 1$ on $\{\theta = 0\}$. Similarly, we have $\cos^2\psi = 1$ on $\{\theta = \pi\}$.
\end{proof}

\subsection{Some sets of $O(2)$-equivariant maps with an isolated singularity}

In the previous subsection, we showed that every $O(2)$-equivariant map $u \in H^1(\Omega,\mathbb{S}^2)$ has a lifting $\psi \in X(D)$ such that $\cos^2 \psi = 1$ along the axis of symmetry in the sense of trace. Since we are mainly concerned with possible models for a point defect, we limit ourself to the following subsets of $X(D)$: For $j \in \ZZ$, define 
\begin{multline*}
X_j(D) = \Big\{\psi \in H^1_{\rm loc}(D) \text{ such that }\\
	  \int_D \Big[(\partial_r \psi)^2   + \frac{1}{r^2} (\partial_\theta \psi)^2 + \frac{1}{r^2 \sin^2 \theta} (\psi - j\theta)^2\Big] r^2 \sin\theta \,d\theta\,dr < \infty \Big\}.
\end{multline*}
Using the inequalities
\begin{align*}
\sin^2\psi
	&\leq \psi^2 \leq 2(\psi - j\theta)^2 + 2j^2\theta^2 \text{ for } 0 \leq \theta \leq \pi/2,\\
\sin^2\psi
	&\leq (\psi - j\pi)^2 \leq 2(\psi - j\theta)^2 + 2j^2(\pi-\theta)^2 \text{ for } \pi/2 < \theta \leq \pi,	
\end{align*}
we see that
\[
X_j(D) \subset X(D) \text{ for all } j \in \ZZ.
\]
It should be clear that if $u \in H^1(\Omega, \mathbb{S}^2) \cap C^1(\Omega \setminus \{0\},\mathbb{S}^2)$ is an $O(2)$-equivariant map, then, up to an addition by a fixed multiple of $\pi$, its lifting $\psi$ belongs to $X_j(D)$ for some $j$.

The following result says that $X_0(D)$ is the completion of $C_c^\infty((0,1] \times (0,\pi))$ under the norm
\[
\|\psi\|_0^2 = \int_D \Big[(\partial_r \psi)^2   + \frac{1}{r^2} (\partial_\theta \psi)^2 + \frac{1}{r^2 \sin^2 \theta} \psi^2\Big] r^2 \sin\theta \,d\theta\,dr
\]
and $X_j(D) = j\theta + X_0(D)$.

\begin{lemma}\label{Lem:XjDen}
$X_0(D)$ is the completion of $C_c^\infty((0,1] \times (0,\pi))$ under the norm $\|\cdot\|_0$.
\end{lemma}

\begin{proof}
It is clear that $C_c^\infty((0,1] \times (0,\pi)) \subset X_0(D)$. Fixing $\psi \in X_0(D)$ and $\vareps > 0$, we proceed to construct $\tilde\psi \in C_c^\infty((0,1] \times (0,\pi))$ such that $\|\psi - \tilde\psi\|_0 \leq \vareps$.

Extending $\psi$ to $D_2 = [0,2] \times [0,\pi]$ by an even reflection across $\{r = 1\}$, we may assume that $\psi$ is defined on $D_2$ and
\begin{equation}
|||\psi|||^2 := \int_{D_2} \Big[(\partial_r \psi)^2   + \frac{1}{r^2} (\partial_\theta \psi)^2 + \frac{1}{r^2 \sin^2 \theta} \psi^2\Big] r^2 \sin\theta \,d\theta\,dr < \infty.
	\label{Eq:Ap1}
\end{equation}
Take a cut-off function $\chi \in C^\infty(\RR)$ such that $\chi \equiv 0$ in $(-1,1)$, $\chi \equiv 1$ in $\RR \setminus (-2,2)$ and $|\chi'| \leq 2$ in $\RR$. For some small $\delta > 0$ to be fixed, let
\begin{align*}
\eta_\delta(r,\theta) &= \chi\big(\frac{r}{\delta}\big) \chi\big(\frac{\theta}{\delta}\big) \chi\big(\frac{\pi - \theta}{\delta}\big) ,\\
\psi_\delta(r,\theta) &= \eta_\delta(r,\theta)  \psi(r,\theta),
\end{align*}
and 
\[
E_\delta = ([0,2\delta] \times [0,\pi]) \cup ([0,2] \times [0,2\delta]) \cup ([0,2] \times [\pi - 2\delta,\pi]).
\]
We have
\begin{align*}
|||\psi - \psi_\delta|||^2
	&= \int_{E_\delta}\Big[((1 - \eta_\delta) \partial_r \psi - \psi \partial_r \eta_\delta)^2    + \frac{1}{r^2} ((1 - \eta_\delta) \partial_\theta \psi - \psi \partial_\theta \eta_\delta)^2 \\
		&\qquad + \frac{1}{r^2 \sin^2 \theta} (1 - \eta_\delta)^2 \psi^2 \Big] r^2 \sin\theta \,d\theta\,dr \\
	&\leq \int_{E_\delta}\Big[2(\partial_r \psi)^2   + \frac{1}{r^2} ( \partial_\theta \psi)^2 + \frac{1}{r^2 \sin^2 \theta}  \psi^2 \Big] r^2 \sin\theta \,d\theta\,dr\\
		&\qquad + \int_{[\delta,2\delta] \times [0,\pi]} 2 \psi^2 (\partial_r \eta_\delta)^2    r^2 \sin\theta \,d\theta\,dr\\
		&\qquad + \int_{[0,2] \times ([\delta,2\delta] \cup [\pi - 2\delta,\pi - \delta]) } 2 \psi^2 (\partial_\theta \eta_\delta)^2    r^2 \sin\theta \,d\theta\,dr\\
	&\leq \int_{E_\delta}\Big[2(\partial_r \psi)^2   + \frac{1}{r^2} ( \partial_\theta \psi)^2 + \frac{1}{r^2 \sin^2 \theta}  \psi^2 \Big] r^2 \sin\theta \,d\theta\,dr\\
		&\qquad + \int_{[\delta,2\delta] \times [0,\pi]} 32 \psi^2   \sin\theta \,d\theta\,dr
			+ \int_{[0,2] \times ([\delta,2\delta] \cup [\pi - 2\delta,\pi - \delta]) } 128 \psi^2   \frac{1}{\sin\theta} \,d\theta\,dr.
\end{align*}
Recalling \eqref{Eq:Ap1}, we may choose $\delta$ sufficiently small such that
\[
|||\psi - \psi_\delta||| \leq \vareps/2.
\]
Noting that the support of $\psi_\delta$ is contained in $D_2 \setminus E_{\delta/2}$, we may then use a standard mollification to obtain $\tilde \psi \in C_c^\infty((0,1] \times (0,\pi))$ such that $\|\psi_\delta - \tilde\psi\|_0 \leq \vareps/2$. The conclusion follows.
\end{proof}

\subsection{Reduced energy for maps with lifting in $X_j(D)$}

We now assume that $u$ is an $O(2)$-equivariant map with a lifting $\psi \in X_j(D)$ and proceed to compute $E[u,\Omega]$. Recall that we assume $k_4 = - k_2$.

We compute
\begin{align*}
\nabla \cdot e_\varphi^{\perp}
	&= \frac{1}{r \sin \theta}, 
	\quad \nabla \cdot e_\varphi = 0,\\
\nabla \times e_\varphi^\perp
	&= 0, 
	\quad \nabla \times e_\varphi = \frac{1}{r \sin\theta} e_3,
\end{align*}
and  
\begin{align*}
\nabla \cdot u
	&=  \cos\psi \sin \theta \partial_r\psi + \frac{1}{r} \cos \psi \cos\theta \partial_\theta \psi    + \frac{1}{r\sin\theta} \sin\psi  \\
		&\qquad- \sin \psi  \cos\theta  \partial_r \psi +  \frac{1}{r} \sin\psi  \sin\theta \partial_\theta \psi \\
	&= - \sin(\psi - \theta) \partial_r\psi +  \frac{1}{r} \cos(\psi - \theta) \partial_\theta  \psi + \frac{1}{r\sin\theta} \sin\psi,\\
\nabla \times u
	&= \cos \psi  \cos\theta \partial_r \psi e_\varphi  - \frac{1}{r}  \cos \psi   \sin\theta  \partial_\theta\psi e_\varphi\\
		&\qquad +  \sin \psi \sin \theta \partial_r \psi e_\varphi +  \frac{1}{r} \sin\psi   \cos\theta  \partial_\theta\psi e_\varphi\\
	&= \Big[\cos(\psi - \theta) \partial_r \psi+  \frac{1}{r} \sin(\psi - \theta)  \partial_\theta \psi\Big]e_\varphi,\\
u \cdot (\nabla \times u)
	&= 0.
\end{align*}
Therefore
\begin{align}
E[u,\Omega]
	&= \pi \int_0^1 \int_0^\pi \Big\{k_1\Big[- \sin(\psi - \theta) \partial_r\psi +  \frac{1}{r} \cos(\psi - \theta) \partial_\theta  \psi + \frac{1}{r\sin\theta} \sin\psi\Big]^2\nonumber\\
		&\qquad + k_3 \Big[\cos(\psi - \theta) \partial_r \psi+  \frac{1}{r} \sin(\psi - \theta)  \partial_\theta \psi\Big]^2\Big\}\,r^2\sin\theta\,d\theta\,dr\nonumber\\
	&= \pi \int_0^1 \int_0^\pi \Big\{
		\big(k_1 \sin^2(\psi - \theta) + k_3 \cos^2(\psi - \theta) \big) (\partial_r\psi)^2\nonumber\\
		&\qquad + \frac{1}{r^2} \big(k_1 \cos^2(\psi - \theta) + k_3 \sin^2(\psi - \theta) \big) (\partial_\theta\psi)^2\nonumber\\
		&\qquad - \frac{2(k_1 - k_3)}{r} \sin(\psi - \theta) \cos(\psi - \theta) \partial_r \psi\partial_\theta \psi\nonumber\\
		&\qquad + \frac{2k_1}{r \sin\theta} \Big[ - \sin(\psi - \theta) \partial_r\psi +  \frac{1}{r} \cos(\psi - \theta) \partial_\theta  \psi\Big] \sin\psi\nonumber\\
		&\qquad + \frac{k_1}{r^2} \csc^2\theta \sin^2\psi
		\Big\}\,r^2\sin\theta\,d\theta\,dr.
	\label{Eq:RE-1}
\end{align}
To handle the term on the second-to-last line, we use the following lemma:

\begin{lemma}\label{Lem:ABCD}
Let
\begin{align*}
(A, B): \RR \times [0,\pi]& \rightarrow  \RR^2 \\
(s,\theta) &\mapsto  (A(s,\theta),B(s,\theta)) 
\end{align*}
be a smooth bounded map satisfying
\[
\partial_\theta B = - A
\]
If $\psi \in X_j(D)$ for some $j \in \ZZ$, then
\[
\int_0^1 \int_0^\pi [r A(\psi,\theta) \partial_r \psi + B(\psi,\theta) \partial_\theta \psi]\,d\theta\,dr
	= \int_0^\pi \int_0^{\psi(1,\theta)} A(s,\theta)\,ds\,d\theta 
		+ \int_0^{j\pi} B(s,\pi)\,ds.
\]
\end{lemma}

\begin{proof}
Using Lemma \ref{Lem:XjDen} and noting that $A$ and $B$ are bounded, we can assume without loss of generality that $\psi - j\theta \in C_c^\infty((0,1] \times (0,\pi))$. In particular,
\[
\psi(r,0) = 0 \text{ and } \psi(r,\pi) = j\pi.
\]

We start by looking for $C$ and $D$ such that
\[
r A(\psi,\theta) \partial_r \psi + B(\psi,\theta) \partial_\theta \psi
	= \partial_r (r C(\psi,\theta)) + \partial_\theta (D(\psi,\theta)).
\]
This leads to the system
\[
\begin{cases}
\partial_s C = A,\\
\partial_s D = B,\\
C = - \partial_\theta D.
\end{cases}
\]
First fix $C$ by
\[
C(s,\theta) = \int_0^s A(\xi,\theta)d\xi
\]
so that $\partial_s C = A$ and $C(0,\cdot) = 0$. It remains to solve
\[
\begin{cases}
\partial_s D = B,\\
\partial_\theta D = - C.
\end{cases}
\]
Since $\RR \times [0,\pi]$ is simply connected and, $\partial_\theta B = -A =  \partial_s (-C)$ (by hypothesis), such $D$ exists uniquely with the normalization condition $D(0,0) = 0$. We note that
\begin{align*}
 D(0,\pi)
	&= D(0,0) + \int_0^\pi \partial_\theta D(0,\theta)d\theta = - \int_0^\pi C(0,\theta)\,d\theta = 0,\\
D(j\pi,\pi) 
	&= D(0,\pi) + \int_0^{j\pi} \partial_s D(s,\pi)\,ds = \int_0^{j\pi} B(s,\pi)\,ds.
\end{align*}
Finally, we compute
\begin{align*}
&\int_0^1 \int_0^\pi [r A(\psi,\theta) \partial_r \psi + B(\psi,\theta) \partial_\theta \psi]\,d\theta\,dr\\
	&\qquad= \int_0^1 \int_0^\pi [\partial_r(r C(\psi,\theta)) +  \partial_\theta (D(\psi,\theta))]\,d\theta\,dr\\
	&\qquad= \int_0^\pi C(\psi(1,\theta),\theta)\,d\theta
		+ \int_0^1 D(\psi(r,\pi),\pi)\,dr - \int_0^1 D(\psi(r,0),0)\,dr \\
	&\qquad= \int_0^\pi \int_0^{\psi(1,\theta)} A(s,\theta)\,ds\,d\theta 
		+ \int_0^{j\pi} B(s,\pi)\,ds.
\end{align*}
The conclusion follows.
\end{proof}

Applying the lemma with 
\[
A(s,\theta) = - \sin(s - \theta) \sin s \quad \text{ and } \quad B(s,\theta) = \cos(s - \theta) \sin s,
\]
we get
\begin{align*}
&\int_0^1 \int_0^\pi   \Big[ - \sin(\psi - \theta) \partial_r\psi +  \frac{1}{r} \cos(\psi - \theta) \partial_\theta  \psi\Big] \sin\psi\,r \,d\theta\,dr\\
	&\qquad=
		- \int_0^\pi \int_0^{\psi(1,\theta)} \sin(s - \theta) \sin s\,ds\,d\theta 
		- \int_0^{j\pi} \cos s \sin s\,ds\\
	&\qquad=
		- \int_0^\pi \int_0^{\psi(1,\theta)} \big( \sin^2 s \cos\theta - \sin s \cos s \sin \theta\big)\,ds\,d\theta 
		 \\
	&\qquad= - \frac{1}{2} \int_0^\pi \Big[ (\psi - \sin\psi \cos\psi)\cos\theta - \sin^2\psi \sin\theta\Big]_{r = 1} \,d\theta.
\end{align*}
Inserting this into \eqref{Eq:RE-1}, we obtain:

\begin{lemma}\label{Lem:RE}
Let $u \in H^1(\Omega,\mathbb{S}^2)$ be an $O(2)$-equivariant map with a lifting $\psi \in X_j(D)$ for some $j \in \ZZ$. Then
\begin{align*}
E[u,\Omega]
	&= J[\psi] := \pi \int_0^1 \int_0^\pi \Big\{
		\big(k_1 \sin^2(\psi - \theta) + k_3 \cos^2(\psi - \theta) \big)  (\partial_r\psi)^2\\
		&\qquad + \frac{1}{r^2} \big(k_1 \cos^2(\psi - \theta) + k_3 \sin^2(\psi - \theta) \big)  (\partial_\theta\psi)^2\\
		&\qquad - \frac{2(k_1 - k_3)}{r} \sin(\psi - \theta) \cos(\psi - \theta) \partial_r \psi\partial_\theta \psi\\
		&\qquad + \frac{k_1}{r^2} \csc^2\theta \sin^2\psi
		\Big\}\,r^2\sin\theta\,d\theta\,dr\\
		&\qquad + k_1 \pi \int_0^\pi \Big[ (\psi - \sin\psi \cos\psi)\cos\theta - \sin^2\psi \sin\theta\Big]_{r = 1} \,d\theta.
\end{align*}
\end{lemma}

\begin{corollary}\label{Cor:EL}
Let $u \in H^1(\Omega,\mathbb{S}^2)$ be an $O(2)$-equivariant map with a lifting $\psi \in X_j(D)$ for some $j \in \ZZ$. Then $u$ is critical for $E$ if and only if $\psi$ satisfies in $D$ the equation
\begin{align}
&-\partial_r \Big\{ r^2 \sin\theta \big(k_1 \sin^2(\psi - \theta) + k_3 \cos^2(\psi - \theta) \big)  \partial_r\psi \nonumber\\
		&\qquad -  (k_1 - k_3) r \sin\theta \sin(\psi - \theta) \cos(\psi - \theta)  \partial_\theta \psi		\Big\}\nonumber\\
	&\quad - \partial_\theta \Big\{
		  \sin\theta \big(k_1 \cos^2(\psi - \theta) + k_3 \sin^2(\psi - \theta) \big)  \partial_\theta\psi\nonumber \\
		&\qquad -  (k_1 - k_3) r \sin\theta \sin(\psi - \theta) \cos(\psi - \theta) \partial_r \psi  
		\Big\}\nonumber\\
	&\quad + \Big\{
		(k_1 - k_3) \sin (\psi - \theta)   \cos (\psi - \theta) \Big[(\partial_r\psi)^2 
			-  \frac{1}{r^2}  (\partial_\theta\psi)^2\Big]\nonumber\\
		&\qquad - \frac{(k_1 - k_3)}{r} \big(\cos^2(\psi - \theta) - \sin^2(\psi - \theta)\big)  \partial_r \psi\partial_\theta \psi\nonumber\\
		&\qquad + \frac{k_1}{r^2} \csc^2\theta \sin\psi \cos\psi
		\Big\}\,r^2\sin\theta
		= 0.
		\label{Eq:EL-1}
\end{align}
\end{corollary}

\begin{proof}
It is straightforward from Lemma \ref{Lem:RE} that if $u$ is critical for $E[u;\Omega]$, then $\psi$ satisfies \eqref{Eq:EL-1} in $D$. Invoking the principal of symmetric criticality \cite{Palais-CMP79} (or by direct computation), if $\psi$ satisfies \eqref{Eq:EL-1} in $D$, then $u$ is critical for $E[u;\Omega\setminus \{x_1 = x_2 = 0\}]$, that is \eqref{Eq:ELFull} holds in $\Omega\setminus \{x_1 = x_2 = 0\}$. Since a line has zero Newtonian capacity, it is standard to show that \eqref{Eq:ELFull} holds in all of $\Omega$.
\end{proof}

\subsection{Zero-homogeneous and $O(2)$-equivariant critical points}

We now suppose that $u$ is a zero-homogeneous and $O(2)$-equivariant critical point of $E$ with a lifting $\psi \in X_j(D)$ for some $j \in \ZZ$. Then $\psi$ is independent of $r$, that is $\psi = \psi(\theta)$. Moreover, by \eqref{Eq:EL-1},
\begin{align*}
&       (k_1 - k_3)   \sin\theta \sin(\psi - \theta) \cos(\psi - \theta)    \psi' \\
	&\quad - \frac{d}{d\theta} \Big\{
		  \sin\theta \big(k_1 \cos^2(\psi - \theta) + k_3 \sin^2(\psi - \theta) \big)   \psi'  \Big\}\\
	&\quad + \Big\{
		-   (k_1 - k_3) \sin (\psi - \theta)   \cos (\psi - \theta)   ( \psi')^2 \\
		&\qquad  + k_1 \csc^2\theta \sin\psi \cos\psi
		\Big\}\, \sin\theta
		= 0 \text{ in } (0,\pi).
\end{align*}
This can be rearranged as
\begin{align}
&\frac{d}{d\theta} \Big[\sin\theta \big(k_1 \cos^2(\psi - \theta) + k_3 \sin^2(\psi - \theta)\big)^{1/2} \psi'\Big]\nonumber\\
	&\qquad
 - k_1 \csc  \theta \big(k_1 \cos^2(\psi - \theta) + k_3 \sin^2(\psi - \theta)\big)^{-1/2} \sin \psi \cos \psi = 0 \text{ in } (0,\pi).
 	\label{Eq:ELy}
\end{align}
It follows that
\begin{align*}
&\frac{d}{d\theta} \Big[\sin^2\theta \big(k_1 \cos^2(\psi - \theta) + k_3 \sin^2(\psi - \theta)\big) ( \psi') ^2\Big]\\
	&\qquad = 2 k_1 \sin \psi \cos \psi \psi' = k_1 \frac{d}{d\theta}\big( \sin^2\psi\big) \text{ in } (0,\pi).
\end{align*}
Therefore, there exists some constant $C$ such that
\begin{equation}
\sin^2\theta \big(k_1 \cos^2(\psi - \theta) + k_3 \sin^2(\psi - \theta)\big) ( \psi') ^2 
	- k_1\sin^2\psi = C \text{ in } (0,\pi).
	\label{Eq:FI-1}
\end{equation}

Noting that the requirement that $\psi \in X(D)$ amounts to
\[
\int_0^\pi \Big[(\psi')^2 + \frac{1}{\sin^2\theta} \sin^2\psi\Big]\,\sin\theta\,d\theta < \infty.
\]
Thus, when dividing the left hand side of \eqref{Eq:FI-1} by $\sin\theta$ we obtain a quantity which is integrable over $(0,\pi)$. This implies that $\frac{C}{\sin\theta}$ is integrable over $(0,\pi)$, and so $C = 0$. We thus have
\begin{equation}
\sin^2\theta \big(k_1 \cos^2(\psi - \theta) + k_3 \sin^2(\psi - \theta)\big) ( \psi') ^2 
	- k_1\sin^2\psi = 0.
	\label{Eq:1stInt}
\end{equation}
We are led to consider the differential equations
\begin{equation}
\psi' = \frac{k_1^{1/2}\sin\psi}{\sin \theta \big(k_1  \cos^2(\psi - \theta) + k_3 \sin^2(\psi - \theta)\big)^{1/2}}
	\label{Eq:1Ix}
\end{equation}
and
\begin{equation}
\psi' = -\frac{k_1^{1/2}\sin\psi}{\sin \theta \big(k_1  \cos^2(\psi - \theta) + k_3 \sin^2(\psi - \theta)\big)^{1/2}}
	\label{Eq:1Iy}
\end{equation}
subjected to an initial value e.g.
\begin{equation}
\psi(\pi/2) = t \in \RR.
	\label{Eq:1II}
\end{equation}
The choice of $\theta = \pi/2$ as an initial point is ad hoc and unimportant. Note that \eqref{Eq:1Iy} can be recast as \eqref{Eq:1Ix} by the change of variable $\psi \mapsto \psi - \pi$.

\begin{proposition}\label{Prop:ODE}
\begin{enumerate}[(i)]
\item If $t \in \pi \ZZ$, then the initial value problem \eqref{Eq:1Ix} and \eqref{Eq:1II} has the unique solution $\psi = t$.

\item If $t \in (2\ell\pi, (2\ell+1)\pi)$ for some $\ell \in \ZZ$, then the initial value problem \eqref{Eq:1Ix} and \eqref{Eq:1II} has the unique solution $\psi \in C^\infty((0,\pi)) \cap C^{2,1}([0,\pi])$, which is strictly increasing and satisfies $\psi(0) = 2\ell\pi$ and $\psi(\pi) = (2\ell+1)\pi$.

\item If $t \in ((2\ell+1)\pi, (2\ell+2)\pi)$ for some $\ell \in \ZZ$, then the initial value problem \eqref{Eq:1Ix} and \eqref{Eq:1II} has the unique solution $\psi \in C^\infty((0,\pi)) \cap C^{2,1}([0,\pi])$, which is strictly decreasing and satisfies $\psi(0) = (2\ell + 2)\pi$ and $\psi(\pi) = (2\ell+1)\pi$.
\end{enumerate}
In all cases, we have
\[
\frac{\sin \psi}{\sin \theta} \in C^{1,1}([0,\pi]).
\]
\end{proposition}

\begin{proof}
Statement (i) is clear. We will only prove statement (ii). The proof of statement (iii) is similar and omitted.

Replacing $\psi$ by $\psi - 2\ell\pi$ if necessary, we may assume without loss of generality that $\ell = 0$. The uniqueness and the smoothness of $\chi$ in its maximal domain of existence follows from standard ODE theory. Let $(c,d) \subset (0,\pi)$ be a maximal interval containing $\pi/2$ on which $\psi$ exists and satisfies $0 < \psi < \pi$. Note that in this interval, $\psi' > 0$ by \eqref{Eq:1Ix}. If $c > 0$, then by standard existence theory, $\psi$ would exist in $(c-\delta,c+\delta)$ for some small $\delta > 0$, and so the maximality of $c$ would imply that $\psi(c) = 0$. This is impossible as uniqueness theory then would give $\psi \equiv 0$, contradicting $\psi(\pi/2) = t \neq 0$. Thus $c = 0$. Arguing similarly, $d = \pi$. We have thus shown that $\psi$ exists and is increasing in $(0,\pi)$ and satisfies $0 < \psi < \pi$ in $(0,\pi)$. 

The rest of the proof concerns the behavior of $\psi$ at the endpoints and is split into three steps.

\medskip
\noindent
\underline{Step 1:} We prove that $\psi \in C^0([0,\pi])$, $\psi(0) = 0$ and $\psi(\pi) = \pi$.

By \eqref{Eq:1Ix}, we have that
\[
\frac{p_1 \sin \psi}{\sin\theta} \leq \psi' \leq \frac{p_2 \sin \psi}{\sin\theta} \quad \text{ in } (0,\pi), 
\]
where
\[
 p_1 := \frac{k_1^{1/2}}{\max\{k_1^{1/2},k_3^{1/2}\}}, p_2 := \frac{k_1^{1/2}}{\min\{k_1^{1/2},k_3^{1/2}\}}.
\]
This implies
\[
p_1\ln \tan\frac{\theta}{2} \leq \ln \tan \frac{\psi}{2} - \ln \tan \frac{t}{2} \leq p_2\ln \tan\frac{\theta}{2} \quad \text{ in } (0,\pi).
\]
Since $\psi$ is increasing and $0 < \psi < \pi$ in $(0,\pi)$, we deduce that $\psi$ is continuous at $\theta = 0$ and $\theta = \pi$ with $\psi(0) = 0$ and $\psi(\pi) = \pi$. 

\medskip
\noindent
\underline{Step 2:} We prove that $\psi \in C^{0,1}([0,\pi])$.

 Fix some small $\vareps \in (0,1/2)$ and $\theta_0 \in (0,\vareps)$ such that $0 < \psi(\theta) < \vareps$ in $(0,\theta_0)$. Using \eqref{Eq:1Ix}, we can find some constant $c_1 > 0$ depending only on $k_1, k_3$ such that
\[
\psi' \geq \frac{(1 - c_1 \vareps^2)\psi}{\theta} \text{ in } (0,\theta_0),
\] 
It follows that
\[
\ln\frac{\psi(\theta_0)}{\psi(\theta)} \geq (1 - c \vareps^2) \ln \frac{\theta_0}{\theta}
\]
and so
\[
\psi(\theta) \leq \psi(\theta_0)\Big(\frac{\theta}{\theta_0}\Big)^{1 - c \vareps^2} \text{ in } (0,\theta_0).
\]
Hence, by sequeezing $\vareps$ if necessary we find some $c_2$ depending on $\psi$ itself such that 
\[
\psi(\theta) \leq c_2 \theta^{1/2}  \text{ in } (0,\theta_0).
\]
Returning once again to \eqref{Eq:1Ix}, we then find a constant $c_3$ depending only on $k_1, k_3$ and $c_2$ such that
\[
\psi' \geq \frac{\sin\psi}{\theta(1 + c_3 \theta)} \text{ in } (0,\theta_0).
\] 
This implies
\[
 \frac{\tan \frac{\psi(\theta_0)}{2}}{\tan \frac{\psi(\theta)}{2}} \geq  \frac{\theta_0(1 + c_3\theta)}{\theta(1 + c_3\theta_0)}\text{ in } (0,\theta_0).
\]
Consequently, we can find a constant $c_4 > 0$ such that
\begin{equation}
\psi(\theta) \leq c_4 \theta \text{ in } (0,\theta_0).
	\label{Eq:psC0-1}
\end{equation}
Arguing similarly, we have that
\begin{equation}
\psi(\theta) \geq \pi - c_5(\pi - \theta) \text{ in } (\pi - \theta_0,\pi).
	\label{Eq:psC0-2}
\end{equation}
Recalling that $0 < \psi < \pi$, we deduce from the last two estimates that
\[
\chi := \frac{\sin \psi}{\sin \theta} \in L^\infty((0,1)).
\] 
Returning to \eqref{Eq:1Ix}, we deduced that $\psi' \in L^\infty((0,\pi))$ and so $\psi \in C^{0,1}([0,\pi])$.

\medskip
\noindent
\underline{Step 3:} We prove that $\psi \in C^{2,1}([0,\pi])$ and $\chi \in C^{1,1}([0,\pi])$.

We compute using \eqref{Eq:1Ix}:
\begin{align*}
\chi' 
	&= \frac{\cos \psi \psi'}{\sin\theta} - \frac{\sin \psi \cos\theta}{\sin^2\theta}\\
	&=  \frac{k_1^{1/2}\sin\psi \cos \psi}{\sin^2 \theta \big(k_1  \cos^2(\psi - \theta) + k_3 \sin^2(\psi - \theta)\big)^{1/2}} - \frac{\sin \psi \cos\theta}{\sin^2\theta}\\
	&= \frac{\chi}{\sin\theta}\Big[\frac{k_1^{1/2} \cos \psi}{  \big(k_1  \cos^2(\psi - \theta) + k_3 \sin^2(\psi - \theta)\big)^{1/2}} -   \cos\theta \Big],
\end{align*}
and
\begin{align*}
\chi''
	&= \frac{\chi'}{\sin\theta}\Big[\frac{k_1^{1/2} \cos \psi}{  \big(k_1  \cos^2(\psi - \theta) + k_3 \sin^2(\psi - \theta)\big)^{1/2}} -   \cos\theta \Big],\\
		&\qquad- \frac{\chi\,\cos\theta}{\sin^2\theta} \Big[\frac{k_1^{1/2} \cos \psi}{  \big(k_1  \cos^2(\psi - \theta) + k_3 \sin^2(\psi - \theta)\big)^{1/2}} -   \cos\theta \Big]\\
		&\qquad +  \chi \Big[\frac{- k_1^{1/2} \chi \psi'}{  \big(k_1  \cos^2(\psi - \theta) + k_3 \sin^2(\psi - \theta)\big)^{1/2}} +  1\Big]\\
		&\qquad +  \frac{\chi}{\sin\theta}  \frac{  (k_1 - k_3) k_1^{1/2} \cos\psi}{  \big(k_1  \cos^2(\psi - \theta) + k_3 \sin^2(\psi - \theta)\big)^{3/2}}  \sin(\psi - \theta) \cos(\psi - \theta).
\end{align*}
Using \eqref{Eq:psC0-1} and \eqref{Eq:psC0-2}, we deduce from the above that 
\[
|\chi'| \leq c_6\sin\theta \text{ in } (0,\pi),
\]
and then
\[
|\chi''| \leq c_7 \text{ in } (0,\pi)
\]
for some constants $c_6, c_7 > 0$. This implies that $\chi \in C^{1,1}([0,\pi])$. By differentiating \eqref{Eq:1Ix}, we have
\begin{align*}
\psi'' 
	&=  \frac{k_1^{1/2}\chi'}{\big(k_1  \cos^2(\psi - \theta) + k_3 \sin^2(\psi - \theta)\big)^{1/2}}\\
		&\qquad +  \frac{(k_1 - k_3) k_1^{1/2}\chi}{ \big(k_1  \cos^2(\psi - \theta) + k_3 \sin^2(\psi - \theta)\big)^{3/2}}\sin(\psi - \theta) \cos(\psi - \theta) \in C^{0,1}([0,\pi])
\end{align*}
and so $\psi \in C^{2,1}([0,\pi])$. The proof is complete.
\end{proof}

We have some immediate consequences of Proposition \ref{Prop:ODE}.

\begin{corollary}\label{Cor:ODE1}
For any $t \in (0,\pi)$, there exists a unique solution $\psi_t \in C^\infty((0,\pi)) \cap C^{2,1}([0,\pi])$ to \eqref{Eq:psitdef}. Moreover, $\psi_t$ is increasing, $\psi_t(0) = 0$, $\psi_t(\pi) = \pi$ and $\frac{\sin\psi_t}{\sin\theta} \in C^{1,1}([0,\pi])$.
\end{corollary}

\begin{corollary}\label{Cor:ODE2}
Let $u$ be a zero-homogeneous and $O(2)$-equivariant critical point of $E$ with a lifting $\psi \in X_j(D)$ for some $j \in \ZZ$. Then $j \in \{0, \pm 1\}$. If $j = 0$, then $u$ is the constant map $u \equiv e_3$. If $j = 1$, then $\psi(\pi/2)  \in (0,\pi)$ and $\psi = \psi_{\psi(\pi/2)}$. If $j = -1$, then $\psi(\pi/2)  \in (-\pi,0)$ and $\psi = - \psi_{-\psi(\pi/2)}$.
\end{corollary}

We conclude this subsection with the following result:
\begin{lemma}\label{Lem:ELValidation}
Let $t \in (0,\pi)$ and $\psi_t$ be the unique solution to \eqref{Eq:psitdef}. Then the map $u_t$ defined by 
\[
u_t(r,\theta,\varphi) = (\sin \psi_t(\theta) \cos \varphi, \sin \psi_t(\theta) \sin \varphi, \cos \psi_t(\theta))^T
\]
belongs to $H^1_{\rm loc}(\RR^3) \cap C^1(\RR^3 \setminus \{0\})$, has a singularity around the origin of topological degree one and is a zero-homogeneous and $O(2)$-equivariant critical points of the Oseen-Frank energy $E$.
\end{lemma}

\begin{proof}
It is easy to see that $u_t \in H^1_{\rm loc}(\RR^3)$. The differentiability of $\psi_t$ in $\RR^3 \setminus \{0\}$ follows from the fact that $\frac{\sin\psi_t}{\sin\theta} \in C^1([0,1])$. It is clear that $u_t$ is zero-homogeneous and $O(2)$-equivariant and has a singularity around the origin of topological degree one. It remains to show that $u_t$ satisfies \eqref{Eq:ELFull}. By Corollary \ref{Cor:EL}, we only need to show that $\psi_t$ satisfies \eqref{Eq:EL-1}, or equivalently \eqref{Eq:ELy}. Indeed, we compute using \eqref{Eq:1Ix}:
\begin{align*}
&\frac{d}{d\theta} \Big[\sin\theta \big(k_1 \cos^2(\psi - \theta) + k_3 \sin^2(\psi - \theta)\big)^{1/2} \psi'\Big]\\
	&\qquad =  \frac{d}{d\theta} [k_1^{1/2} \sin \psi]  
		= k_1^{1/2} \cos \psi \psi'\\
	&= k_1 \csc  \theta \big(k_1 \cos^2(\psi - \theta) + k_3 \sin^2(\psi - \theta)\big)^{-1/2} \sin \psi \cos \psi \text{ in } (0,\pi).
\end{align*}
This clearly gives \eqref{Eq:ELy}.
\end{proof}

\subsection{Minimality of $\psi_t$ in $X_1(D)$}

In this subsection, we prove:

\begin{proposition}\label{Prop:Min}
Let $t \in (0,\pi)$ and $\psi_t$ be the solution to \eqref{Eq:psitdef}. Then
\[
J[\psi_t] \leq J[\psi] \text{ for all } \psi \in X_1(D) \text{ satisfying } \psi|_{r = 1} = \psi_t,
\]
where equality is attained if and only if $\psi = \psi_t$. Here $J$ is the energy derived in Lemma \ref{Lem:RE}.
\end{proposition}

\begin{proof}
We compute
\begin{align*}
J[\psi]
	&= \pi \int_0^1 \int_0^\pi \Big\{
		\big(k_1 \sin^2(\psi - \theta) + k_3 \cos^2(\psi - \theta) \big) \times\\
			&\qquad\qquad 
			 \times \Big[	  \partial_r\psi - \frac{k_1 - k_3}{r} \frac{\sin(\psi - \theta) \cos(\psi - \theta)}{ k_1 \sin^2(\psi - \theta) + k_3 \cos^2(\psi - \theta)} \times \\
			&\qquad \qquad \qquad 
			\times \Big(\partial_\theta \psi -   \frac{k_1^{1/2}}{\big(k_1 \cos^2(\psi - \theta) + k_3 \sin^2(\psi - \theta)\big)^{1/2} } \frac{\sin\psi}{\sin \theta}\Big)\Big]^2 \\
		&\qquad + \frac{k_1k_3}{r^2} \frac{1}{k_1 \sin^2(\psi - \theta) + k_3 \cos^2(\psi - \theta)}  \times\\
			&\qquad\qquad \times \Big(\partial_\theta \psi -   \frac{k_1^{1/2}}{\big(k_1 \cos^2(\psi - \theta) + k_3 \sin^2(\psi - \theta)\big)^{1/2} } \frac{\sin\psi}{\sin \theta}\Big)^2\\
		&\qquad - \frac{2(k_1 - k_3)k_1^{1/2}}{r} \frac{ \sin(\psi - \theta) \cos(\psi - \theta) }{\big(k_1 \cos^2(\psi - \theta) + k_3 \sin^2(\psi - \theta)\big)^{1/2} } \frac{\sin\psi}{\sin \theta} \partial_r \psi\\
		&\qquad + \frac{2k_1^{1/2}}{r^2} \frac{k_1 k_3 + (k_1 - k_3)^2 \sin^2(\psi - \theta) \cos^2(\psi - \theta)}{ k_1 \sin^2(\psi - \theta) + k_3 \cos^2(\psi - \theta) } \times\\
			&\qquad\qquad \times \frac{1}{\big(k_1 \cos^2(\psi - \theta) + k_3 \sin^2(\psi - \theta)\big)^{1/2} } \frac{\sin\psi}{\sin \theta}\partial_\theta \psi 
		\Big\}\,r^2\sin\theta\,d\theta\,dr\\
		&\quad + k_1 \pi \int_0^\pi \Big[ (\psi - \sin\psi \cos\psi)\cos\theta - \sin^2\psi \sin\theta\Big]_{r = 1} \,d\theta.
\end{align*}
Now, applying Lemma \ref{Lem:ABCD}  with
\begin{align*}
A(s,\theta)
	&= -   \frac{ (k_1 - k_3 )\sin(s - \theta) \cos(s - \theta) }{\big(k_1 \cos^2(s - \theta) + k_3 \sin^2(s - \theta)\big)^{1/2} }  \sin s,\\
B(s,\theta)
	&= \frac{k_1 k_3 + (k_1 - k_3)^2 \sin^2(s - \theta) \cos^2(s - \theta)}{ k_1 \sin^2(s - \theta) + k_3 \cos^2(s - \theta) } \times\\
			&\qquad\qquad \times \frac{1}{\big(k_1 \cos^2(s - \theta) + k_3 \sin^2(s- \theta)\big)^{1/2} }  \sin s,
\end{align*}
we get
\begin{align*}
&\int_0^1 \int_0^\pi \Big\{  -  \frac{r (k_1 - k_3)\sin(\psi - \theta) \cos(\psi - \theta) }{\big(k_1 \cos^2(\psi - \theta) + k_3 \sin^2(\psi - \theta)\big)^{1/2} } \sin\psi  \partial_r \psi\\
		&\qquad +   \frac{k_1 k_3 + (k_1 - k_3)^2 \sin^2(\psi - \theta) \cos^2(\psi - \theta)}{ k_1 \sin^2(\psi - \theta) + k_3 \cos^2(\psi - \theta) } \times\\
			&\qquad\qquad \times \frac{1}{\big(k_1 \cos^2(\psi - \theta) + k_3 \sin^2(\psi - \theta)\big)^{1/2} }  \sin\psi \partial_\theta \psi 
		\Big\} \,d\theta\,dr\\
	&\quad = - (k_1 - k_3 ) \int_0^\pi \int_0^{\psi(1,\theta)}      \frac{ \sin(s - \theta) \cos(s - \theta) }{\big(k_1 \cos^2(s - \theta) + k_3 \sin^2(s - \theta)\big)^{1/2} }  \sin s\,ds\,d\theta \\
		&\qquad
		+ \int_0^{ \pi} \frac{k_1 k_3 + (k_1 - k_3)^2 \sin^2s \cos^2s}{ k_1 \sin^2 s + k_3 \cos^2s}  \frac{1}{\big(k_1 \cos^2s + k_3 \sin^2s\big)^{1/2} }  \sin s\,ds
\end{align*}
Altogether we have
\begin{align*}
J[\psi]
	&= \pi \int_0^1 \int_0^\pi \Big\{
		\big(k_1 \sin^2(\psi - \theta) + k_3 \cos^2(\psi - \theta) \big) \times \\
			&\qquad\qquad 
			\times\Big[
		  \partial_r\psi - \frac{k_1 - k_3}{r} \frac{\sin(\psi - \theta) \cos(\psi - \theta)}{ k_1 \sin^2(\psi - \theta) + k_3 \cos^2(\psi - \theta)} \times \\
			&\qquad \qquad \qquad 
			\times \Big(\partial_\theta \psi -   \frac{k_1^{1/2}}{\big(k_1 \cos^2(\psi - \theta) + k_3 \sin^2(\psi - \theta)\big)^{1/2} } \frac{\sin\psi}{\sin \theta}\Big)\Big]^2 \\
		&\qquad + \frac{k_1k_3}{r^2} \frac{1}{k_1 \sin^2(\psi - \theta) + k_3 \cos^2(\psi - \theta)}  \times\\
			&\qquad\qquad \times \Big(\partial_\theta \psi -   \frac{k_1^{1/2}}{\big(k_1 \cos^2(\psi - \theta) + k_3 \sin^2(\psi - \theta)\big)^{1/2} } \frac{\sin\psi}{\sin \theta}\Big)^2 \Big\}\,r^2\sin\theta\,d\theta\,dr\\
		&\quad + k_1 \pi \int_0^\pi \Big[ (\psi - \sin\psi \cos\psi)\cos\theta - \sin^2\psi \sin\theta\Big]_{r = 1} \,d\theta\\
		&\quad   - 2k_1^{1/2} (k_1 - k_3 )\pi \int_0^\pi \int_0^{\psi(1,\theta)}      \frac{ \sin(s - \theta) \cos(s - \theta) }{\big(k_1 \cos^2(s - \theta) + k_3 \sin^2(s - \theta)\big)^{1/2} }  \sin s\,ds\,d\theta \\
		&\quad
		+ 2k_1^{1/2}\pi \int_0^{\pi} \frac{k_1 k_3 + (k_1 - k_3)^2 \sin^2s \cos^2s}{ k_1 \sin^2 s + k_3 \cos^2s}  \frac{1}{\big(k_1 \cos^2s + k_3 \sin^2s\big)^{1/2} }  \sin s\,ds.
\end{align*}
Using $\psi = \psi_t$ on $\{r = 1\}$ and \eqref{Eq:1Ix}, we see that $J[\psi_t]$ is equal to the sum on the last three lines and yielding
\[
J[\psi] \geq J[\psi_t].
\]
Moreover, if equality holds, we must have
\begin{align*}
&\partial_r\psi 
	- \frac{k_1 - k_3}{r} \frac{\sin(\psi - \theta) \cos(\psi - \theta)}{ k_1 \sin^2(\psi - \theta) + k_3 \cos^2(\psi - \theta)} \times \\
		&\qquad \qquad 
			\times \Big(\partial_\theta \psi -   \frac{k_1^{1/2}}{\big(k_1 \cos^2(\psi - \theta) + k_3 \sin^2(\psi - \theta)\big)^{1/2} } \frac{\sin\psi}{\sin \theta}\Big) = 0,\\
&\partial_\theta \psi -   \frac{k_1^{1/2}}{\big(k_1 \cos^2(\psi - \theta) + k_3 \sin^2(\psi - \theta)\big)^{1/2} } \frac{\sin\psi}{\sin \theta}
	= 0,
\end{align*}
which clearly implies $\partial_r \psi = 0$ and hence $\psi = \psi_t$ (since they agree on $\{r = 1\}$).
\end{proof}

\subsection{Proof of the main theorem}

\begin{proof}[Proof of Theorem \ref{Thm:Main}]
Statement (i) follows from Corollary \ref{Cor:ODE1}.

The first half of statement (ii) follows Lemma \ref{Lem:ELValidation}. The second half of statement (ii) follows from Corollary \ref{Cor:ODE2}.

Statement (iii) follows from the uniqueness for $\psi_t$ and the facts that
\begin{itemize}
\item $\psi(\theta) = \theta$ solves \eqref{Eq:psitdef} with $t = \pi/2$ for any $k_1, k_3 > 0$,
\item $\psi(\theta) = \arccos \frac{\cos t + \cos \theta}{1 + \cos t \cos\theta}$ solves \eqref{Eq:psitdef} when $k_1 = k_3$.
\end{itemize}

Statement (iv) follows from Proposition \ref{Prop:Min}.
\end{proof}

\def\cprime{$'$}

\end{document}